\documentclass[a4paper,12pt,reqno]{amsart}

\usepackage{amsmath}
\usepackage{amssymb}
\usepackage{amsfonts}
\usepackage{graphicx}
\usepackage[colorlinks]{hyperref}
\renewcommand\eqref[1]{(\ref{#1})} %Need with hyperref

\graphicspath{ {images/} }
\title[Geometric Hardy and Hardy-Sobolev inequalities]{Geometric Hardy and Hardy-Sobolev inequalities on Heisenberg groups}
\author[Michael Ruzhansky]{Michael Ruzhansky}
\address{\href{www.ruzhansky.org}{Michael Ruzhansky:}
	\endgraf
	Department of Mathematics 
	\endgraf
	Ghent University, Belgium
	\endgraf
	and
	\endgraf
	School of Mathematical Sciences
		\endgraf Queen Mary University of London 
			\endgraf
		United Kingdom
			\endgraf
	{\it E-mail address} {\rm Michael.Ruzhansky@ugent.be}
}

\author[Bolys Sabitbek]{Bolys Sabitbek}
\address{ \href{https://www.researchgate.net/profile/Bolys_Sabitbek3}{Bolys Sabitbek:}
	\endgraf
	Institute of Mathematics and Mathematical Modeling 
	\endgraf
	125 Pushkin Street., Almaty, 050010
	\endgraf
	Kazakhstan
	\endgraf
	and
	\endgraf 
	Department of Mechanics and Mathematics
	\endgraf 
	Al-Farabi Kazakh National University 
	\endgraf
	71 al-Farabi Ave., Almaty, 050040 
	\endgraf Kazakhstan
	\endgraf
	{\it E-mail address} {\rm b.sabitbek@math.kz}
}

\author[Durvudkhan Suragan]{Durvudkhan Suragan}
\address{\href{https://sst.nu.edu.kz/en/durvudkhan-suragan-phd/}{Durvudkhan Suragan:}
	\endgraf
	Department of Mathematics
	\endgraf
	Nazarbayev University
	\endgraf
	53 Kabanbay batyr Ave., Astana, 010000
	\endgraf
	Kazakhstan
	\endgraf
	{\it E-mail address} {\rm durvudkhan.suragan@nu.edu.kz}
}


\thanks{The first
 author was supported by the EPSRC Grant 
EP/R003025/1, by the Leverhulme Research Grant RPG-2017-151, and by the FWO Odysseus grant. The second author was supported by Nazarbayev University Faculty Development Competitive Research Grants N090118FD5342. The third author was supported in parts by the MESRK grant AP05130981. No new data was collected or generated during the course of this research.}

\newtheoremstyle{theorem}%name
{10pt}          % space above
{10pt}  % space below
{\sl}  % bofy font
{\parindent}     % ident - empty=no indent,  \parindent= paragraph indent
{\bf}  % thm head font
{. }    % punctuation after thm head
{ }    % space after thm head: `` ``=normal \newline=linebreak
{}     % thm head specification
\theoremstyle{theorem}

\numberwithin{equation}{section}
\theoremstyle{plain}
\newtheorem{thm}{Theorem}[section]

\newtheorem{cor}[thm]{Corollary}
\newtheorem{lem}[thm]{Lemma}

\theoremstyle{definition}

\newtheorem{rem}[thm]{Remark}

\newcommand{\G}{\mathbb{G}}

\newcommand{\La}{\mathcal{L}}

\newtheoremstyle{defi}%name
{10pt}          % space above
{10pt}  % space below
{\rm}  % bofy font
{\parindent}     % ident - empty=no indent,  \parindent= paragraph indent
{\bf}  % thm head font
{. }    % punctuation after thm head
{ }    % space after thm head: `` ``=normal \newline=linebreak
{}     % thm head specification
\theoremstyle{defi}

\begin{document}
		\begin{abstract}
		In this paper, we present the geometric Hardy inequality for the sub-Laplacian in the half-spaces on the stratified groups. As a consequence, we obtain the following geometric Hardy inequality in a half-space on the Heisenberg group with a sharp constant
		\begin{equation*}
		\int_{\mathbb{H}^+} |\nabla_{H}u|^p d\xi \geq \left(\frac{p-1}{p}\right)^p \int_{\mathbb{H}^+}  \frac{\mathcal{W}(\xi)^p }{dist(\xi,\partial \mathbb{H}^+)^p} |u|^p d\xi, \,\, p>1,
		\end{equation*}
		which solves the conjecture in the paper \cite{Larson}. Also, we obtain a version of the Hardy-Sobolev inequality in a half-space on the Heisenberg group
		\begin{equation*}
		\left( \int_{\mathbb{H}^+} |\nabla_{H} u|^p d\xi - \left(\frac{p-1}{p}\right)^p \int_{\mathbb{H}^+} \frac{ \mathcal{W}(\xi)^p}{dist(\xi,\partial \mathbb{H}^+)^p} |u|^p d\xi \right)^{\frac{1}{p}} \geq C \left( \int_{\mathbb{H}^+} |u|^{p^*} d\xi\right)^{\frac{1}{p^*}},
		\end{equation*}
		where $dist(\xi,\partial \mathbb{H}^+)$ is the Euclidean distance to the boundary, $p^* := Qp/(Q-p)$, $2\leq p<Q$, and     
		$$\mathcal{W}(\xi)=\left(\sum_{i=1}^{n}\langle X_i(\xi), \nu \rangle^2+\langle Y_i(\xi), \nu \rangle^2\right)^{\frac{1}{2}},$$
		is the angle function. For $p=2$, this gives the Hardy-Sobolev-Maz'ya inequality on the Heisenberg group.  
	\end{abstract}
	\maketitle
\section{Introduction}
Let us recall the Hardy inequality in the half-space 
\begin{equation}\label{HS1}
\int_{\mathbb{R}^n_+} |\nabla u|^p dx \geq \left( \frac{p-1}{p}\right)^p \int_{\mathbb{R}^n_+} \frac{|u|^p}{x_n^p}dx,\quad p>1,
\end{equation}
for every function $u \in C_0^{\infty}(\mathbb{R}^n_+)$, where $\nabla$ is the usual Euclidean gradient. The half-space is defined by $\mathbb{R}^n_+:= \{(x',x_n)|x':=(x_1,\ldots,x_{n-1})\in\mathbb{R}^{n-1}, x_n>0 \},\, n \in \mathbb{N}.$ There is a number of studies related to inequality \eqref{HS1}, see e.g. \cite{Ancona},  \cite{Avka_Lap}, \cite{Davies}, \cite{Opic-Kuf}.

Filippas, Maz'ya and Tertikas in \cite{FMT} have established the Hardy-Sobolev inequality in the following form
\begin{equation}\label{HS2}
\left(\int_{\mathbb{R}^n_+} |\nabla u|^p dx - \left( \frac{p-1}{p}\right)^p \int_{\mathbb{R}^n_+} \frac{|u|^p}{x_n^p}dx\right)^{\frac{1}{p}} \geq C \left( \int_{\mathbb{R}^n_+} |u|^{p^*} dx\right)^{ \frac{1}{p^*}},
\end{equation}
for all function $u \in C_0^{\infty}(\mathbb{R}^n_+)$, where $p^*=\frac{np}{n-p}$ and $2\leq p<n$. There is a different proof of this inequality by Frank and Loss \cite{FL}.

The Hardy inequality in the half-space on the Heisenberg group was shown by Luan and Young \cite{LY} in the form
\begin{equation}\label{LY_ineq}
\int_{\mathbb{H}^+} |\nabla_{H}u|^2 d\xi \geq \frac{1}{4} \int_{\mathbb{H}^+}  \frac{|x|^2+|y|^2 }{t^2} |u|^2 d\xi.
\end{equation} 
An alternative proof of this inequality was given by Larson in \cite{Larson}, where the author generalised it to any half-space of the Heisenberg group, 
\begin{equation*}
	\int_{\mathbb{H}^+} |\nabla_{H}u|^2 d\xi \geq \frac{1}{4} \int_{\mathbb{H}^+}  \frac{\sum_{i=1}^{n}\langle X_i(\xi), \nu \rangle^2+\langle Y_i(\xi), \nu \rangle^2 }{dist(\xi,\partial \mathbb{H}^+)^2} |u|^2 d\xi,
\end{equation*}
where $X_i$ and $Y_i$ (for $i=1,\ldots,n$) are left-invariant vector fields on the Heisenberg group, $\nu$ is the Riemannian outer unit normal (see \cite{Garofalo}) to the boundary. Also, there is the $L^p$-generalisation of the above inequality
\begin{equation*}
\int_{\mathbb{H}^+} |\nabla_{H}u|^p d\xi \geq \left(\frac{p-1}{p}\right)^p \int_{\mathbb{H}^+}  \frac{\sum_{i=1}^{n}|\langle X_i(\xi), \nu \rangle|^p+|\langle Y_i(\xi), \nu \rangle|^p }{dist(\xi,\partial \mathbb{H}^+)^p} |u|^p d\xi.
\end{equation*} 
Note also that the authors in \cite{RSS_geom} have extended this result to general Carnot groups.

The main aim of this paper is to improve the $L^p$-version of the geometric Hardy inequality for the sub-Laplacian in the half-spaces on the stratified groups (the Carnot groups), where the obtained inequality will be a natural extension of the inequality derived by the authors in \cite{Larson} and \cite{RSS_geom} on the Heisenberg and the stratified groups, respectively. Note that it proves the inequality conjectured in \cite{Larson}.  Moreover, we obtain a version of the Hardy-Sobolev inequality in the setting of the Heisenberg group.  

The main results of this paper are as follows:
\begin{itemize}
	\item \textbf{Geometric $L^p$-Hardy inequality on $\G^+$:} 	Let $\G^+:= \{ x \in \G: \langle x, \nu \rangle > d  \}$ be a half-space of a stratified group $\G$. Then for all $u \in C^{\infty}_0(\G^+)$, $\beta \in \mathbb{R}$ and $p>1$ we have
	\begin{align*}
	\int_{\G^+} |\nabla_{\G}u|^p dx \geq& -(p-1)(|\beta|^{\frac{p}{p-1}}+\beta) \int_{\G^+} \frac{\mathcal{W}(x)^p}{dist(x,\partial \G^+)^p} |u|^p dx \\
	&+\beta\int_{\G^+} \frac{\La_p (dist(x,\partial \G^+))}{dist(x,\partial \G^+)^{p-1}} |u|^p dx,\nonumber
	\end{align*}
	where $\mathcal{W}(x):=(\sum_{i=1}^{N}\langle X_i(x),\nu\rangle^2)^{1/2}$, and $\mathcal{L}_p$ is the $p$-sub-Laplacian on $\G$, see \eqref{Lp}. 
	\item  \textbf{Geometric $L^p$-Hardy inequality on $\mathbb{H}^+$:} Let $\mathbb{H}^+:= \{\xi \in \mathbb{H}^n : \langle \xi , \nu \rangle >d \}$ be a half-space of the Heisenberg group. Then for all $u \in C^{\infty}_0(\mathbb{H}^+)$ and $p> 1$ we have
	\begin{equation*}
	\int_{\mathbb{H}^+} |\nabla_{H}u|^p d\xi \geq  \left(\frac{p-1}{p}\right)^p \int_{\mathbb{H}^+} \frac{\mathcal{W}(\xi)^p}{dist(\xi,\partial \mathbb{H}^+)^p} |u|^p d\xi,
	\end{equation*}
	where $\mathcal{W}(\xi):=\left( \sum_{i=1}^{n}\langle X_i(\xi),\nu\rangle^2+\langle Y_i(\xi),\nu\rangle^2\right)^{1/2}$ and the constant is sharp.
	\item \textbf{Geometric Hardy-Sobolev inequality on $\mathbb{H}^+$:}	Let $\mathbb{H}^+:= \{\xi \in \mathbb{H}^n:\langle \xi,\nu \rangle >d \}$ be a half-space of the Heisenberg group. Then for all $u \in C_0^{\infty}(\mathbb{H}^+)$ and $2\leq p<Q$ with $Q=2n+1$, there exists some $C>0$ such that we have 
	\begin{equation*}
	\left( \int_{\mathbb{H}^+} |\nabla_{H} u|^p d\xi - \left(\frac{p-1}{p}\right)^p \int_{\mathbb{H}^+} \frac{ \mathcal{W}(\xi)^p}{dist(\xi,\partial \mathbb{H}^+)^p} |u|^p d\xi \right)^{\frac{1}{p}} \geq C \left( \int_{\mathbb{H}^+} |u|^{p^*} d\xi\right)^{\frac{1}{p^*}},
	\end{equation*}
	where $dist(\xi,\partial \mathbb{H}^+) := \langle \xi, \nu \rangle - d$ is the distance from $\xi$ to the boundary and $p^* := Qp/(Q-p)$.
\end{itemize}

\subsection{Preliminaries on stratified groups}    
Let $\mathbb{G}=(\mathbb{R}^n,\circ,\delta_{\lambda})$ be a stratified Lie group (or a homogeneous Carnot group), with dilation structure $\delta_{\lambda}$ and Jacobian generators $X_{1},\ldots,X_{N}$, so that $N$ is the dimension of the first stratum of $\mathbb{G}$. Let us denote by $Q$ the homogeneous dimension of $\G$.  We refer to the recent books \cite{FR} and \cite{RS_book} for extensive discussions of stratified Lie groups and their properties.

The sub-Laplacian on $\mathbb{G}$ is given by
\begin{equation}\label{sublap}
\mathcal{L}=\sum_{k=1}^{N}X_{k}^{2}.
\end{equation}
We also recall that the standard Lebesque measure $dx$ on $\mathbb R^{n}$ is the Haar measure for $\mathbb{G}$ (see, e.g. \cite[Proposition 1.6.6]{FR}).
Each left invariant vector field $X_{k}$ has an explicit form and satisfies the divergence theorem,
see e.g. \cite{FR} for the derivation of the exact formula: more precisely, we can express
\begin{equation}\label{Xk0}
X_{k}=\frac{\partial}{\partial x'_{k}}+
\sum_{l=2}^{r}\sum_{m=1}^{N_{l}}a_{k,m}^{(l)}(x',...,x^{(l-1)})
\frac{\partial}{\partial x_{m}^{(l)}},
\end{equation}
with $x=(x',x^{(2)},\ldots,x^{(r)})$, where $r$ is the step of $\G$ and
$x^{(l)}=(x^{(l)}_1,\ldots,x^{(l)}_{N_l})$ are the variables in the $l^{th}$ stratum,
see also \cite[Section 3.1.5]{FR} for a general presentation.
The horizontal gradient is given by
$$\nabla_{\G}:=(X_{1},\ldots, X_{N}),$$
and the horizontal divergence is defined by
$${\rm div}_{\G} v:=\nabla_{\G}\cdot v.$$
The $p$-sub-Laplacian has the following form
\begin{equation}\label{Lp}
	\mathcal{L}_p v = \nabla_{\G} (|\nabla_{\G} v|^{p-2}\nabla_{\G} v).
\end{equation}
Let us define the half-space on the stratified group $\G$ as
\begin{equation*}
\G^+ := \{ x \in \G: \langle x, \nu \rangle > d  \},
\end{equation*}
where $\nu:=(\nu_1,\ldots,\nu_r)$ with $\nu_j \in \mathbb{R}^{N_j},\,\, j=1,\ldots,r,$ is the Riemannian outer unit normal to $\partial \G^+$ (see \cite{Garofalo}) and $d \in \mathbb{R}$. Let us define the so-called angle function  
\begin{equation*}
	\mathcal{W}(x) := \sqrt{ \sum_{i=1}^{N}\langle X_i(x),\nu\rangle^2};
\end{equation*} 
such function was introduced by Garofalo \cite{Garofalo} in his investigation of the horizontal Gauss map.
The Euclidean distance to the boundary $\partial \G^+$ is denoted by $dist(x,\partial \G^+)$ and defined as
\begin{equation*}
dist(x,\partial \G^+)= \langle x, \nu \rangle - d.
\end{equation*} 
\section{Geometric Hardy inequalities}
\begin{thm}\label{thm1}
	Let $\G^+$ be a half-space of a stratified group $\G$. Then for all $\beta \in \mathbb{R}$ and $p>1$ we have
	\begin{align}\label{Hardy_stratified}
	\int_{\G^+} |\nabla_{\G}u|^p dx \geq& -(p-1)(|\beta|^{\frac{p}{p-1}}+\beta) \int_{\G^+} \frac{\mathcal{W}(x)^p}{dist(x,\partial \G^+)^p} |u|^p dx \\
	&+\beta\int_{\G^+} \frac{\La_p(dist(x,\partial \G^+))}{dist(x,\partial \G^+)^{p-1}} |u|^p dx,\nonumber 
	\end{align}
	for all $u \in C^{\infty}_0(\G^+)$.
\end{thm}
\begin{proof}[Proof of Theorem \ref{thm1}]
	Let us begin with the divergence theorem, then we apply the H\"older inequality and the Young inequality, respectively. It follows for a vector field $V \in C^{\infty}(\G^+) $ that
\begin{align*}
		\int_{\G^+} { \rm div}_{\G}V  |u|^p dx &= -p \int_{\G^+}|u|^{p-1}  \langle V, \nabla_{\G} u \rangle dx \\
		& \leq p \left( \int_{\G^+} |\nabla_{\G} u|^p dx\right)^{\frac{1}{p}} \left( \int_{\G^+} |V|^{\frac{p}{p-1}}|u|^p dx \right)^{\frac{p-1}{p}}\\
		& \leq \int_{\G^+} |\nabla_{\G} u|^p dx + (p-1)\int_{\G^+} |V|^{\frac{p}{p-1}}|u|^p dx.	\end{align*}
		By rearranging the above expression, we arrive at
	\begin{equation}\label{1}
		\int_{\G^+} |\nabla_{\G} u|^p dx \geq \int_{\G^+} ({\rm div}_{\G}V-(p-1)|V|^{\frac{p}{p-1}} )|u|^p dx.
	\end{equation}
	Now we choose $V$ in the following form
	\begin{equation}\label{2}
	V = \beta \frac{|\nabla_{\G} dist(x,\partial \G^+)|^{p-2}}{dist(x,\partial \G^+)^{p-1}}\nabla_{\G} dist(x,\partial \G^+),
	\end{equation}
that is
	\begin{equation*}
	|V|^{\frac{p}{p-1}} =|\beta|^{\frac{p}{p-1}} \frac{|\nabla_{\G} dist(x,\partial \G^+)|^p}{dist(x,\partial \G^+)^p}.
	\end{equation*}
	Also, we have
	\begin{align*}
	|\nabla_{\G} dist(x,\partial \G^+)|^p &= \left|\left( X_1dist(x,\partial \G^+),\ldots,X_N dist(x,\partial \G^+) \right)\right|^p\\
	& =|\left( \langle X_1(x),\nu\rangle,\ldots,\langle X_N(x),\nu\rangle  \right)|^p\\
	& =\left( \sum_{i=1}^{N}\langle X_i(x),\nu \rangle^2 \right)^{\frac{p}{2}}= \mathcal{W}(x)^p.
	\end{align*}
	Indeed, let us show that $\langle X_i(x),\nu \rangle = X_i\langle x,\nu \rangle$:
	\begin{multline*}
	X_i(x) = (\overset{i}{\overbrace{0,\ldots,1},\ldots,0}, \underset{N_2}{\underbrace{a_{i,1}^{(2)}(x'),\ldots,a_{i,N_2}^{(2)}(x')}},\ldots, \\ \underset{N_r}{\underbrace{a_{i,1}^{(r)}(x',x^{(2)},\ldots,x^{(r-1)}),\ldots a_{i,N_r}^{(r)}(x',x^{(2)},\ldots,x^{(r-1)})}}),
	\end{multline*}
	\begin{equation*}
		\langle X_i(x), \nu \rangle = \nu'_i + \sum_{l=2}^{r}\sum_{m=1}^{N_l}a_{i,m}^{(l)}(x',x^{(2)},\ldots,x^{(r-1)})\nu^{(l)}_m,
	\end{equation*}
 and 
	\begin{align*}
		&\langle x,\nu \rangle = \sum_{k=1}^{N}x_k'\nu'_k + \sum_{l=2}^{r}\sum_{m=1}^{N_l}x_m^{(l)}\nu_m^{(l)},\\
		&X_i\langle x,\nu \rangle= \nu'_i + \sum_{l=2}^{r}\sum_{m=1}^{N_l}a_{i,m}^{(l)}(x',x^{(2)},\ldots,x^{(r-1)})\nu^{(l)}_m.
	\end{align*}
A direct calculation shows that
	\begin{align*}
		 { \rm div}_{\G} V =& \beta \frac{\nabla_{\G}(|\nabla_{\G}dist(x,\partial \G^+)|^{p-2}\nabla_{\G}dist(x,\partial \G^+))}{dist(x,\partial \G^+)^{p-1}}\\
		 & - \beta(p-1) \frac{|\nabla_{\G}dist(x,\partial \G^+)|^{p-2}\nabla_{\G}dist(x,\partial \G^+)dist(x,\partial \G^+)^{p-2}\nabla_{\G}dist(x,\partial \G^+)}{dist(x,\partial \G^+)^{2(p-1)}}\\
		 =&\beta \frac{\La_p (dist(x,\partial \G^+))}{dist(x,\partial \G^+)^{p-1}} - \beta(p-1)\frac{|\nabla_{\G} dist(x,\partial \G^+)|^p}{dist(x,\partial \G^+)^p}.
	\end{align*}
	So we get 
\begin{align*}
	{\rm div}_{\G}V-(p-1)|V|^{\frac{p}{p-1}}  =& - (p-1)(|\beta|^{\frac{p}{p-1}}+\beta)\frac{|\nabla_{\G} dist(x,\partial \G^+)|^p}{dist(x,\partial \G^+)^p}\\
	& + \beta \frac{\La_p(dist(x,\partial \G^+))}{dist(x,\partial \G^+)^{p-1}}.
\end{align*}	
	Putting the above expression into inequality \eqref{1},  we arrive at 
	 \begin{align*}
	 \int_{\G^+} |\nabla_{\G} u|^pdx &\geq -(p-1)(|\beta|^{\frac{p}{p-1}}+\beta) \int_{\G^+} \frac{(\sum_{i=1}^{N}\langle X_i(x),\nu\rangle^2)^{\frac{p}{2}}}{dist(x,\partial \G^+)^p} |u|^p dx \\
	 &+\beta\int_{\G^+} \frac{\La_p(dist(x,\partial \G^+))}{dist(x,\partial \G^+)^{p-1}} |u|^p dx,
	 \end{align*}
	 completing the proof. 
\end{proof}	
\subsection{Preliminaries on the Heisenberg group} Let us give a brief introduction of the Heisenberg group.
Let $\mathbb{H}^n$ be the Heisenberg group, that is, the set $\mathbb{R}^{2n+1}$ equipped with the group law 
\begin{equation*}
\xi \circ \widetilde{\xi} := (x + \widetilde{x}, y + \widetilde{y}, t + \widetilde{t}+2 \sum_{i=1}^{n}(\widetilde{x}_i y_i - x_i \widetilde{y}_i )),
\end{equation*}
where $\xi:= (x,y,t) \in \mathbb{H}^n$, $x:=(x_1,\ldots,x_n)$, $y:=(y_1,\ldots,y_n)$, and $\xi^{-1}=-\xi$ is the inverse element of $\xi$ with respect to the group law. The dilation operation of the Heisenberg group with respect to the group law has the following form
\begin{equation*}
\delta_{\lambda}(\xi) := (\lambda x, \lambda y, \lambda^2 t) \,\, \text{for}\,\, \lambda>0.
\end{equation*} 
The Lie algebra $\mathfrak{h}$ of the left-invariant vector fields on the Heisenberg group $\mathbb{H}^n$ is spanned by 
\begin{equation*}
X_i:= \frac{\partial }{\partial x_i} + 2y_i\frac{\partial }{\partial t} \,\, \text{for} \,\, 1\leq i \leq n,
\end{equation*}
\begin{equation*}
Y_i:= \frac{\partial }{\partial y_i} - 2x_i\frac{\partial }{\partial t} \,\, \text{for} \,\, 1\leq i \leq n,
\end{equation*}
and  with their (non-zero) commutator
\begin{equation*}
[X_i,Y_i]= - 4 \frac{\partial}{\partial t}.
\end{equation*} 
The horizontal gradient of $\mathbb{H}^n$ is given by 
\begin{equation*}
\nabla_{H}:= (X_1,\ldots,X_n,Y_1,\ldots,Y_n),
\end{equation*} 
so the sub-Laplacian on $\mathbb{H}^n$ is given by
\begin{equation*}
\mathcal{L}:= \sum_{i=1}^{n} \left(X_i^2 + Y_i^2\right). 
\end{equation*} 

Let us define the half-space of the Heisenberg group by
\begin{equation*}
\mathbb{H}^+:= \{\xi \in \mathbb{H}^n : \langle \xi , \nu \rangle >d \},
\end{equation*}
where $\nu:= (\nu_x,\nu_y,\nu_t)$ with $\nu_x,\nu_y \in \mathbb{R}^n$ and $\nu_t \in \mathbb{R}$ is the Riemannian outer unit normal to $\partial \mathbb{H}^+$ (see \cite{Garofalo}) and $d \in \mathbb{R}$. Let us define the so-called angle function  
\begin{equation*}
\mathcal{W}(\xi) := \sqrt{ \sum_{i=1}^{n}\langle X_i(\xi),\nu\rangle^2+\langle Y_i(\xi),\nu\rangle^2}.
\end{equation*}  The Euclidean distance to the boundary $\partial \mathbb{H}^+$ is denoted by $dist(\xi,\partial \mathbb{H}^+)$ and defined by
\begin{equation*}
dist(\xi,\partial \mathbb{H}^+) := \langle \xi, \nu \rangle - d.
\end{equation*}
\subsection{Consequences on the Heisenberg group}
As a consequence of Theorem \ref{thm1} we have the following inequality.
\begin{cor}\label{cor12}
	Let $\mathbb{H}^+$ be a half-space of the Heisenberg group $\mathbb{H}^n$. Then for all $u \in C^{\infty}_0(\mathbb{H}^+)$ and $p> 1$ we have
	\begin{equation}\label{Hardy1}
	\int_{\mathbb{H}^+} |\nabla_{H}u|^p d\xi \geq  \left(\frac{p-1}{p}\right)^p \int_{\mathbb{H}^+} \frac{\mathcal{W}(\xi)^p}{dist(\xi,\partial \mathbb{H}^+)^p} |u|^p d\xi,
	\end{equation}
	where the constant is sharp.
\end{cor}
\begin{rem}
	Note that inequality \eqref{Hardy1} was conjectured in \cite{Larson} which is a natural extension of inequality \eqref{LY_ineq} in \cite{LY}. Also, the sharpness of inequality \eqref{Hardy1} was proved by choosing $\nu:=(1,0,\ldots,0)$ and $d=0$ in the paper \cite{Larson}.
\end{rem}

\begin{proof}[Proof of Corollary \ref{cor12}]
Let us rewrite the inequality in Theorem \ref{thm1} in terms of the Heisenberg group as follows
	\begin{align*}
\int_{\mathbb{H}^+} |\nabla_{H} u|^p d\xi \geq& -(p-1)(|\beta|^{\frac{p}{p-1}}+\beta) \int_{\mathbb{H}^+} \frac{\mathcal{W}(\xi)^p}{dist(\xi,\partial \mathbb{H}^+)^p} |u|^p d\xi \\
&+\beta\int_{\mathbb{H}^+} \frac{\La_p(dist(\xi,\partial\mathbb{H}^+))}{dist(\xi,\partial\mathbb{H}^+)^{p-1}}|u|^p d\xi.
\end{align*}
In the case of the Heisenberg group, we need to show that the last term vanishes to prove Corollary \ref{cor12}. Indeed, we have
\begin{equation*}
\La_p(dist(\xi,\partial\mathbb{H}^+))=0,
\end{equation*} 
since 
\begin{align*}
	&\langle X_i(\xi),\nu\rangle= \nu_{x,i}+2y_i\nu_t, & &\langle Y_i(\xi),\nu\rangle= \nu_{y,i}-2x_i\nu_t, \\
	& X_i\langle X_i(\xi),\nu\rangle=0,& & Y_i\langle Y_i(\xi),\nu\rangle=0,\\
	& Y_i\langle X_i(\xi),\nu\rangle=2 \nu_t, && X_i\langle Y_i(\xi),\nu\rangle=-2 \nu_t,
\end{align*}
where $\xi:=(x,y,t)$ with $x,y \in \mathbb{R}^n$ and $t \in \mathbb{R}$, $\nu := (\nu_x,\nu_y,\nu_t)$ with $\nu_x:=(\nu_{x,1},\ldots \nu_{x,n})$ and $\nu_y:=(\nu_{y,1},\ldots \nu_{y,n})$. Then we have 
	\begin{align*}
&X_i(\xi) = (\underset{n}{\underbrace{\overset{i}{\overbrace{0,\ldots,1},\ldots,0}}},\underset{n}{\underbrace{0,\ldots,0}}, 2y_i),\\
&Y_i(\xi) = (\underset{n}{\underbrace{0,\ldots,0}},\underset{n}{\underbrace{\overset{i}{\overbrace{0,\ldots,1},\ldots,0}}},- 2x_i).
\end{align*}
So we have
	\begin{equation*}
		\int_{\mathbb{H}^+} |\nabla_{H} u|^p d\xi \geq -(p-1)(|\beta|^{\frac{p}{p-1}}+\beta) \int_{\mathbb{H}^+} \frac{\mathcal{W}(\xi)^p}{dist(\xi,\partial \mathbb{H}^+)^p} |u|^p d\xi.
	\end{equation*}
Now we optimise by differentiating the above inequality with respect to $\beta$, so that we have
	\begin{equation*}
		\frac{p}{p-1}|\beta|^{\frac{1}{p-1}}+1=0,
	\end{equation*}
	which leads to 
	\begin{equation*}
		\beta = -  \left(\frac{p-1}{p}\right)^{p-1}.
	\end{equation*}
	Using this value of $\beta$, we arrive at 
		\begin{equation*}
	\int_{\mathbb{H}^+} |\nabla_{H} u|^p d\xi \geq \left(\frac{p-1}{p}\right)^{p} \int_{\mathbb{H}^+} \frac{\mathcal{W}(\xi)^p}{dist(\xi,\partial \mathbb{H}^+)^p} |u|^p d\xi.
	\end{equation*}
	We have finished the proof of Corollary \ref{cor12}.
\end{proof}
\section{Geometric Hardy-Sobolev inequalities}
In this section, we present the geometric Hardy-Sobolev inequality in the half space on the Heisenberg group.

\subsection{A lower estimate for the geometric Hardy type inequalities}
We start with an estimate for the remainder.
\begin{lem}\label{lem1}
	Let $\mathbb{H}^+$ be a half-space of the Heisenberg group $\mathbb{H}^n$. Then for $p\geq 2$, there exists a constant $C_p > 0$ such that
	\begin{align}
	E_p[u]=& \int_{\mathbb{H}^+} |\nabla_{H}u|^p d\xi - \left(\frac{p-1}{p}\right)^p \int_{\mathbb{H}^+} \frac{ |\nabla_{H}dist(\xi,\partial \mathbb{H}^+)|^p}{dist(\xi,\partial \mathbb{H}^+)^p} |u|^p d\xi \nonumber\\
	\geq& C_p  \int_{\mathbb{H}^+} |dist(\xi,\partial \mathbb{H}^+)|^{p-1}|\nabla_{H} v|^p d\xi,
	\end{align}
	for all $u \in C_0^{\infty}(\mathbb{H}^+)$, where $dist(\xi,\partial \mathbb{H}^+) := \langle \xi, \nu \rangle - d$ is the distance from $\xi$ to the boundary, $C_p=(2^{p-1}-1)^{-1}$, and $u(\xi)=dist(\xi, \partial \mathbb{H}^+)^{\frac{p-1}{p}}v(\xi)$.
\end{lem}
The Euclidean version of such a lower estimate to the Hardy inequality was established by Barbaris, Filippas and Tertikas \cite{BFT}.
\begin{proof}[Proof of Lemma \ref{lem1}]
	Let us begin by recalling once again the angle function, denoted by $\mathcal{W}$,
	\begin{align}\label{est}
	|\nabla_{H}dist(\xi,\partial \mathbb{H}^+)|^p &= |(X_1\langle \xi, \nu\rangle,\ldots,X_n\langle \xi,  \nu\rangle, Y_1\langle \xi, \nu\rangle,\ldots, Y_n\langle \xi, \nu\rangle)|^p \nonumber\\
	& = \left( \sum_{i=1}^{n} \langle X_i(\xi), \nu\rangle^2 + \langle Y_i(\xi), \nu\rangle^2 \right)^{\frac{p}{2}} \nonumber\\
	&= \mathcal{W}(\xi)^p.
	\end{align}
	Note that $X_i\langle \xi,\nu \rangle$ is equal to $\langle X_i(\xi),\nu \rangle$, see the proof of Theorem \ref{thm1}. This expression $|\nabla_{H}dist(\xi,\partial \mathbb{H}^+)|^p =  \mathcal{W}(\xi)^p$
	will be used later. For now we will estimate the following form
	\begin{align}\label{est1}
	E_p[u]:= \int_{\mathbb{H}^+} |\nabla_{H}u|^p d\xi - \left(\frac{p-1}{p}\right)^p \int_{\mathbb{H}^+} \frac{ \mathcal{W}(\xi)^p}{dist(\xi,\partial \mathbb{H}^+)^p} |u|^p d\xi.
	\end{align}
	To estimate this, we introduce the following ground transform 
	\begin{equation}
	u(\xi) = dist(\xi,\partial \mathbb{H}^+)^{\frac{p-1}{p}}v(\xi).
	\end{equation}
	By inserting it into \eqref{est1} and using \eqref{est}, we have
	\begin{align*}
	E_p[u]=& \int_{\mathbb{H}^+} \left| \frac{p-1}{p} dist(\xi,\partial \mathbb{H}^+)^{-\frac{1}{p}} \nabla_{H}dist(\xi,\partial \mathbb{H}^+) v + dist(\xi,\partial \mathbb{H}^+)^{\frac{p-1}{p}}\nabla_{H} v \right|^p d\xi \\
	& - \left(\frac{p-1}{p}\right)^p \int_{\mathbb{H}^+}\frac{ \mathcal{W}(\xi)^p}{dist(\xi,\partial \mathbb{H}^+)^p}|dist(\xi,\partial \mathbb{H}^+)^{\frac{p-1}{p}}v|^p d\xi \\
	\geq& \int_{\mathbb{H}^+} \left| \frac{p-1}{p} dist(\xi,\partial \mathbb{H}^+)^{-\frac{1}{p}}  v + dist(\xi,\partial \mathbb{H}^+)^{\frac{p-1}{p}} \frac{\nabla_{H} v}{\nabla_{H}dist(\xi,\partial \mathbb{H}^+)} \right|^p|\mathcal{W}(\xi)|^p \\
	&- \left| \frac{p-1}{p} dist(\xi,\partial \mathbb{H}^+)^{-\frac{1}{p}}  v \right|^p |\mathcal{W}(\xi)|^p  d\xi.  
	\end{align*}
	Then for $p\geq 2$ and $A,B \in \mathbb{R}^n$ we have that
	\begin{equation*}
	|A + B|^p - |A|^p \geq C_p |B|^p + p|A|^{p-2}A\cdot B, 
	\end{equation*} 
	where $C_p = (2^{p-1}-1)^{-1}$ (see \cite[Lemma 3.3]{BFT}). By taking 
	\begin{equation*}
		A := \frac{p-1}{p} dist(\xi,\partial \mathbb{H}^+)^{-\frac{1}{p}}  v \,\,\, \text{and} \,\,\, B:=dist(\xi,\partial \mathbb{H}^+)^{\frac{p-1}{p}}\frac{\nabla_{H} v}{\nabla_{H}dist(\xi,\partial \mathbb{H}^+)},
	\end{equation*}
	then we have the following lower estimate
	\begin{align*}
	E_p[u]  & \geq \int_{\mathbb{H}^+} |\mathcal{W}(\xi)|^p(|A+B|^p - |A|^p) d\xi  \\
	&\geq C_p  \int_{\mathbb{H}^+} dist(\xi,\partial \mathbb{H}^+)^{p-1}|\nabla_{H} v|^p \frac{\mathcal{W}(\xi)^p}{|\nabla_{H} dist(\xi,\partial \mathbb{H}^+)|^p} d\xi \\
	& + \left(\frac{p-1}{p}\right)^{p-1} \int_{\mathbb{H}^+}|\mathcal{W}(\xi)|^p|\nabla_{H} dist(\xi,\partial \mathbb{H}^+)|^{p-2} \left(\nabla_{H}dist(\xi,\partial \mathbb{H}^+) \cdot \nabla_{H} |v|^p\right) d\xi \\ 
	& \geq C_p \int_{\mathbb{H}^+} dist(\xi,\partial \mathbb{H}^+)^{p-1}|\nabla_{H} v|^p d\xi.
	\end{align*}
	In the last line we have used \eqref{est} and we dropped the last term on the right-hand side. This completes the proof of Lemma \ref{lem1}.
\end{proof}
\subsection{Geometric Hardy-Sobolev inequality in the half-space on $\mathbb{H}^n$}
Now we are ready to obtain the geometric Hardy-Sobolev inequality in the half-space on the Heisenberg group $\mathbb{H}^n$.
\begin{thm}\label{thm} 
	Let $\mathbb{H}^+$ be a half-space of the Heisenberg group $\mathbb{H}^n$. Then for every function $u \in C_0^{\infty}(\mathbb{H}^+)$ and $2\leq p<Q$ with $Q=2n+1$, there exists some $C>0$ such that we have
	\begin{equation}
	\left( \int_{\mathbb{H}^+} |\nabla_{H} u|^p d\xi - \left(\frac{p-1}{p}\right)^p \int_{\mathbb{H}^+} \frac{ \mathcal{W}(\xi)^p}{dist(\xi,\partial \mathbb{H}^+)^p} |u|^p d\xi \right)^{\frac{1}{p}} \geq C \left( \int_{\mathbb{H}^+} |u|^{p^*} d\xi\right)^{\frac{1}{p^*}},
	\end{equation}
	where $p^* := Qp/(Q-p)$ and $dist(\xi,\partial \mathbb{H}^+) := \langle \xi, \nu \rangle - d$ is the distance from $\xi$ to the boundary.
\end{thm}
\begin{rem}
	Note that for $p=2$ we have the Hardy-Sobolev-Maz'ya inequality in the following form
	\begin{equation}\label{eq1}
	\left( \int_{\mathbb{H}^+} |\nabla_{H} u|^2 d\xi - \frac{1}{4}\int_{\mathbb{H}^+} \frac{ \mathcal{W}(\xi)^2}{dist(\xi,\partial \mathbb{H}^+)^2} |u|^2 d\xi \right)^{\frac{1}{2}} \geq C \left( \int_{\mathbb{H}^+} |u|^{2^*} d\xi\right)^{\frac{1}{2^*}},
	\end{equation}
	where $2^*:=2Q/(Q-2).$
\end{rem}

\begin{proof}[Proof of Theorem \ref{thm}]
	Our key ingredient of proving the Hardy-Sobolev inequality in the half-space of $\mathbb{H}^n$ is the $L^1$-Sobolev inequality, or the Gagliardo-Nirenberg inequality. It has been established on the Heisenberg group by Baldi, Franchi, Pansu in \cite{BFP}. 
	
	The $L^1$-Sobolev inequality on the Heisenberg group follows in the form
	\begin{equation*}
	c \left( \int_{\mathbb{H}^n}|g|^{\frac{Q}{Q-1}} d\xi \right)^{\frac{Q-1}{Q}} \leq \int_{\mathbb{H}^n} |\nabla_{H} g| d\xi,
	\end{equation*}
	for some $c>0$, for every function $g \in W^{1,1}(\mathbb{H}^n)$. Now let us set $g = |u|^{p^*(1-1/Q)}$, then we obtain
	\begin{align*}
	c \left( \int_{\mathbb{H}^+}|u|^{p^*} d\xi\right)^{\frac{Q-1}{Q}} &\leq \left|\frac{p(Q-1)}{Q-p} \right| \int_{\mathbb{H}^+} |u|^{\frac{Q(p-1)}{Q-p}} |\nabla_{H} |u|| d\xi,\\ \nonumber
	& \leq \left|\frac{p(Q-1)}{Q-p} \right| \int_{\mathbb{H}^+} |u|^{\frac{Qp}{Q-p}\frac{(p-1)}{p}} |\nabla_{H} u| d\xi, \\
	& = \left|\frac{p(Q-1)}{Q-p} \right| \int_{\mathbb{H}^+} |u|^{p^*(1-1/p)}|\nabla_{H} u| d\xi. 
	\end{align*}
	We have used $|\nabla_{H}|u||\leq |\nabla_{H} u|$ (see \cite[Proof of Thoerem 2.1]{RSS_Lp}). Then we arrive at 
	\begin{equation}\label{est3}
		C_1 \left( \int_{\mathbb{H}^+}|u|^{p^*} d\xi\right)^{\frac{Q-1}{Q}} \leq  \int_{\mathbb{H}^+} |u|^{p^*(1-1/p)}|\nabla_{H} u| d\xi,
	\end{equation}
	where $C_1:=c\left|\frac{Q-p}{p(Q-1)}\right|>0$. Let us estimate the right-hand side of inequality \eqref{est3}. Again we use a ground transform $u(\xi) = dist(\xi,\partial \mathbb{H}^+)^{\frac{p-1}{p}}v(\xi)$ which leads to 
	\begin{align*}
	&\int_{\mathbb{H}^+} |u|^{p^*(1-1/p)} |\nabla_{H} u| d\xi\\
	& = \int_{\mathbb{H}^+} |u|^{p^*(1-1/p)} \left| dist(\xi,\partial \mathbb{H}^+)^{\frac{p-1}{p}} \nabla_{H}v  + \frac{p-1}{p} dist(\xi,\partial \mathbb{H}^+)^{-1/p} \nabla_{H}dist(\xi,\partial \mathbb{H}^+) v \right| d\xi \\
	& \leq \int_{\mathbb{H}^+} |u|^{p^*(1-1/p)} dist(\xi,\partial \mathbb{H}^+)^{\frac{p-1}{p}} |\nabla_{H} v| d\xi \\
	&+ \frac{p-1}{p} \int_{\mathbb{H}^+} dist(\xi,\partial \mathbb{H}^+)^{p^*(1-1/p)^2-1/p} |\nabla_{H} dist(\xi,\partial \mathbb{H}^+)| |v|^{p^*(1-1/p)+1} d\xi \\
	& = I_1 +  \frac{p-1}{p} I_2.
	\end{align*} 
	In the last line we have denoted two integrals by $I_1$ and $I_2$, respectively. Also, for simplification we denote $\alpha:= p^*(1-1/p)^2 +1-1/p$.
	First, we estimate $I_2$ using integrations by parts 
	\begin{align*}
	I_2 &= \int_{\mathbb{H}^+} dist(\xi,\partial \mathbb{H}^+)^{\alpha-1} |\nabla_{H}dist(\xi,\partial \mathbb{H}^+)| |v|^{\alpha p/(p-1)} d\xi \\
	& = \frac{1}{\alpha}  \int_{\mathbb{H}^+} \langle \nabla_{H} dist(\xi,\partial \mathbb{H}^+)^{\alpha}, \nabla_{H}dist(\xi,\partial \mathbb{H}^+)\rangle  \frac{|v|^{\alpha p/(p-1)}}{|\nabla_{H}dist(\xi,\partial \mathbb{H}^+)|}  d\xi\\
	& = -\frac{1}{\alpha} \int_{\mathbb{H}^+} dist(\xi,\partial \mathbb{H}^+)^{\alpha}  \nabla_{H} \left(  \frac{\nabla_{H}dist(\xi,\partial \mathbb{H}^+)|v|^{\alpha p/(p-1)}}{|\nabla_{H}dist(\xi,\partial \mathbb{H}^+)|}   \right)d\xi \\
	&=  -\frac{1}{\alpha} \int_{\mathbb{H}^+} dist(\xi,\partial \mathbb{H}^+)^{\alpha} \times \\
	& \left( \frac{\nabla_{H}dist(\xi,\partial \mathbb{H}^+)\nabla_{H}|v|^{\alpha p/(p-1)}}{|\nabla_{H}dist(\xi,\partial \mathbb{H}^+)|} 
	- \frac{\langle \nabla_{H}dist(\xi,\partial \mathbb{H}^+),\nabla_{H}|\nabla_{H}dist(\xi,\partial \mathbb{H}^+)|\rangle |v|^{\alpha p/(p-1)}}{|\nabla_{H}dist(\xi,\partial \mathbb{H}^+)|^2} \right) d\xi\\
	& =  -\frac{1}{\alpha} \int_{\mathbb{H}^+} dist(\xi,\partial \mathbb{H}^+)^{\alpha} \frac{\nabla_{H}dist(\xi,\partial \mathbb{H}^+)\nabla_{H}|v|^{\alpha p/(p-1)}}{|\nabla_{H}dist(\xi,\partial \mathbb{H}^+)|} d\xi\\	
	& \leq -\frac{p}{p-1} \int_{\mathbb{H}^+} dist(\xi,\partial \mathbb{H}^+)^{\alpha} |v|^{\alpha p/(p-1)-1} |\nabla_{H} v|d\xi\\
	& = -\frac{p}{p-1} \int_{\mathbb{H}^+}|u|^{p^*(1-1/p)} dist(\xi,\partial \mathbb{H}^+)^{\frac{p-1}{p}} |\nabla_{H}v| d\xi\\
	& \leq \frac{p}{p-1} I_1.
	\end{align*}
	We have used $|\nabla_{H}|u||\leq |\nabla_{H}u|$, and
	\begin{align*}
		\mathcal{L}dist(\xi,\partial \mathbb{H}^+)= \sum_{i=1}^{n}X_i\langle X_i(\xi),\nu \rangle + Y_i\langle Y_i(\xi),\nu \rangle =0,
	\end{align*} 
	since 
	\begin{align*}
	&\langle X_i(\xi),\nu\rangle= \nu_{x,i}+2y_i\nu_t, & &\langle Y_i(\xi),\nu\rangle= \nu_{y,i}-2x_i\nu_t, \\
	& X_i\langle X_i(\xi),\nu\rangle=0,& & Y_i\langle Y_i(\xi),\nu\rangle=0,\\
	& Y_i\langle X_i(\xi),\nu\rangle=2 \nu_t, && X_i\langle Y_i(\xi),\nu\rangle=-2 \nu_t,
	\end{align*}
	where $\xi:=(x,y,t)$ with $x,y \in \mathbb{R}^n$ and $t \in \mathbb{R}$, $\nu := (\nu_x,\nu_y,\nu_t)$ with $\nu_x:=(\nu_{x,1},\ldots \nu_{x,n})$ and $\nu_y:=(\nu_{y,1},\ldots \nu_{y,n})$.
	Also we have
	\begin{multline*}
		\langle \nabla_{H}dist(\xi,\partial \mathbb{H}^+),\nabla_{H}|\nabla_{H}dist(\xi,\partial \mathbb{H}^+)|\rangle= \frac{2\nu_t}{|\nabla_{H}dist(\xi,\partial \mathbb{H}^+)|}\\
		 ( \underset{n}{\underbrace{(2x_1\nu_t -\nu_{y,1})(\nu_{x,1}+2y_1\nu_t)+\ldots+(2x_n\nu_t -\nu_{y,n})(\nu_{x,n}+2y_n\nu_t)}} \\ 
		 +\underset{n}{\underbrace{(\nu_{y,1} -2x_1\nu_t )(\nu_{x,1}+2y_1\nu_t)+\ldots+(\nu_{y,n}-2x_n\nu_t )(\nu_{x,n}+2y_n\nu_t)}}  )=0,
	\end{multline*}
	since
	\begin{multline*}
		\nabla_{H}|\nabla_{H}dist(\xi,\partial \mathbb{H}^+)|=(X_1|\nabla_{H}dist(\xi,\partial \mathbb{H}^+)|,\ldots, X_n|\nabla_{H}dist(\xi,\partial \mathbb{H}^+)|,\\Y_1|\nabla_{H}dist(\xi,\partial \mathbb{H}^+)|,\ldots,Y_n|\nabla_{H}dist(\xi,\partial \mathbb{H}^+)|)\\
		= \frac{2\nu_t}{|\nabla_{H}dist(\xi,\partial \mathbb{H}^+)|}(\underset{n}{\underbrace{2x_1\nu_t-\nu_{y,1},\ldots, 2x_n\nu_t-\nu_{y,n}}},\underset{n}{\underbrace{\nu_{x,1}+2y_1\nu_t,\ldots,\nu_{x,n}+2y_n\nu_t}}),
	\end{multline*}
	and 
	\begin{equation*}
		\nabla_{H}dist(\xi,\partial \mathbb{H}^+)= (\underset{n}{\underbrace{\nu_{x,1}+2y_1\nu_t,\ldots,\nu_{x,n}+2y_n\nu_t}},\underset{n}{\underbrace{\nu_{y,1}-2x_1\nu_t,\ldots,\nu_{y,n}-2x_n\nu_t}}).
	\end{equation*}
 As we see that integral $I_2$ can be estimated by integral $I_1$. From this estimation we know that 
	\begin{equation}\label{est2}
	\int_{\mathbb{H}^+} |u|^{p^*(1-1/p)} |\nabla_{H} u| d\xi \leq 2I_1.
	\end{equation}
	Now it comes to estimate $I_1$ by using the H\"older inequality 
	\begin{align*}
	I_1 &= \int_{\mathbb{H}^+} \left\{ |u|^{p^*(1-1/p)} \right\} \left\{dist(\xi,\partial \mathbb{H}^+)^{(p-1)/p} |\nabla_{H} v| \right\} d\xi \\
	& \leq \left( \int_{\mathbb{H}^+} |u|^{p^*} d\xi\right)^{1-1/p}  \left(  \int_{\mathbb{H}^+} dist(\xi,\partial \mathbb{H}^+)^{p-1} |\nabla_{H} v|^p  d\xi  \right)^{1/p}\\
	& \leq C_p^{-1/p} \left( \int_{\mathbb{H}^+} |u|^{p^*} d\xi\right)^{1-1/p}  \left(\int_{\mathbb{H}^+} |\nabla_{H}u|^p d\xi - \left(\frac{p-1}{p}\right)^p \int_{\mathbb{H}^+} \frac{ \mathcal{W}(\xi)^p}{dist(\xi,\partial \mathbb{H}^+)^p} |u|^p d\xi\right)^{1/p}.
	\end{align*}
	In the last line we have used Lemma \ref{lem1}.
	Inserting the estimate of $I_1$ in \eqref{est2}, we arrive at 
	\begin{align*}
	&\int_{\mathbb{H}^+} |u|^{p^*(1-1/p)}|\nabla_{H} u| d\xi \leq \\
	&2C_p^{-1/p} \left( \int_{\mathbb{H}^+} |u|^{p^*} d\xi\right)^{1-1/p}  \left(\int_{\mathbb{H}^+} |\nabla_{H}u|^p d\xi - \left(\frac{p-1}{p}\right)^p \int_{\mathbb{H}^+} \frac{ \mathcal{W}(\xi)^p}{dist(\xi,\partial \mathbb{H}^+)^p} |u|^p d\xi\right)^{\frac{1}{p}}.
	\end{align*}
	Plugging the above estimate in \eqref{est3}, we have 
	\begin{align*}
	C_1 \left( \int_{\mathbb{H}^+}|u|^{p^*} d\xi\right)^{\frac{Q-1}{Q}} &\leq 2C_p^{-1/p} \left( \int_{\mathbb{H}^+} |u|^{p^*} d\xi\right)^{\frac{p-1}{p}}   \\
	&\left(\int_{\mathbb{H}^+} |\nabla_{H}u|^p d\xi - \left(\frac{p-1}{p}\right)^p \int_{\mathbb{H}^+} \frac{ \mathcal{W}(\xi)^p}{dist(\xi,\partial \mathbb{H}^+)^p} |u|^p d\xi\right)^{\frac{1}{p}}.
	\end{align*}
	By collecting terms, we finish the proof of Theorem \ref{thm}.
\end{proof}
\subsection{Consequence of Theorem \ref{thm}} 
Let us demonstrate our result in a particular case when $p=2$:

\begin{cor}\label{cor1}
	Let $\mathbb{H}^+:=\{ \xi =(x,y,t) \in \mathbb{H}^n | \,\, t>0 \}$ be a half-space of the Heisenberg group $\mathbb{H}^n$. Then for every function $u \in C_0^{\infty}(\mathbb{H}^+)$ taking $d=0$ we have 
	\begin{equation}\label{3.10}
	\left( \int_{\mathbb{H}^+} |\nabla_{H} u|^2 d\xi - \int_{\mathbb{H}^+} \frac{ |x|^2 + |y|^2}{t^2} |u|^2 d\xi \right)^{\frac{1}{2}} \geq C \left( \int_{\mathbb{H}^+} |u|^{2^*} d\xi\right)^{\frac{1}{2^*}},
	\end{equation}
	where $2^*:=2Q/(Q-2), \,\, Q=2n+2$, with $C>0$ independent of $u$. 
\end{cor}
\begin{proof}[Proof of Corollary \ref{cor1}]
	We have the following left-invariant vector fields 
	\begin{align*}
	X_i = \frac{\partial }{\partial x_i} + 2y_i \frac{\partial }{\partial t} \,\,\, \text{and} \,\,\,
	 Y_i = \frac{\partial }{\partial y_i} - 2x_i \frac{\partial }{\partial t},	
	\end{align*}
	with the commutator 
	\begin{equation*}
	[X_i,Y_i] = -4 \frac{\partial }{\partial t}.
	\end{equation*}
	Then for $\xi = (x_1,\ldots,x_n,y_1,\ldots,y_n,t)$ and $\nu=(\overset{n}{\overbrace{0,\ldots,0}},\overset{n}{\overbrace{0,\ldots,0}},1)$, we get
	\begin{align*}
	\langle X_i(\xi), \nu \rangle = 2y_i \,\,\, \text{and}\,\,\, \langle Y_i(\xi), \nu \rangle = -2x_i,
	\end{align*}
	where 
	\begin{align*}
 &X_i(\xi) = (\overset{i}{\overbrace{0,\ldots,1}},\ldots,0,\overset{n}{\overbrace{0,\ldots,0}},2y_i), \\ 
 & Y_i(\xi) = (\overset{n}{\overbrace{0,\ldots,0}},\overset{i}{\overbrace{0,\ldots,1}},\ldots,0,,-2x_i).
	\end{align*}
	Thus, we arrive at
	\begin{equation}
	\frac{\mathcal{W}(\xi)^2}{dist(\xi,\partial \mathbb{H}^+)^2} = 4\frac{|x|^2+|y|^2}{t^2}.
	\end{equation}
	Plugging the above expression into inequality \eqref{eq1} we obtain
	\begin{equation*}
	\left( \int_{\mathbb{H}^+} |\nabla_{H} u|^2 d\xi - \int_{\mathbb{H}^+} \frac{ |x|^2 + |y|^2}{t^2} |u|^2 d\xi \right)^{\frac{1}{2}} \geq C \left( \int_{\mathbb{H}^+} |u|^{2^*} d\xi\right)^{\frac{1}{2^*}},
	\end{equation*} 
	showing \eqref{3.10}.
\end{proof}

\end{document}